\documentclass[12pt]{amsart} 
 
\usepackage[dvips]{graphics} 
\usepackage{color,amssymb,amsbsy,amsmath,amsfonts,amssymb,
amscd,epsfig,times,graphics}

\newtheorem{theorem}{Theorem}[section] 
\newtheorem{proposition}[theorem]{Proposition}
\newtheorem{example}[theorem]{Example}
\newtheorem{lemma}[theorem]{Lemma} 

\newtheorem{remark}[theorem]{Remark}
\theoremstyle{definition}
\newtheorem{definition}[theorem]{Definition}
 
\numberwithin{equation}{section}

\newcommand{\C}{\mathbb C} 
\newcommand{\R}{\mathbb R} 
\newcommand{\N}{\mathbb N} 
\newcommand{\B}{\mathbb B} 
\newcommand{\D}{\mathbb D}

\begin{document} 
\begin{abstract}
We present different constructions of abstract boundaries for bounded complete (Kobayashi) hyperbolic domains in $\C^d$, $d \geq 1$. These constructions essentially come from the geometric theory of metric spaces. We also present, as an application, some extension results concerning biholomorphic maps.
\end{abstract} 

\title{Abstract boundaries and continuous  extension of biholomorphisms}

\dedicatory{To Professor L\'aszl\'o Lempert for his 70th birthday}

\author[F. Bracci \and H. Gaussier]{Filippo Bracci$^1$ and Herv\'e Gaussier$^2$}

\address{\begin{tabular}{lll}
 Filippo Bracci & & Herv\'e Gaussier\\
Dipartimento di Matematica & & Univ. Grenoble Alpes, IF, F-38000 Grenoble, France\\
Universit\`a di Roma ``Tor Vergata''& & CNRS, IF, F-38000 Grenoble, France\\
Via della Ricerca Scientifica, 1& &
 \\
00133 Roma& & 
\\
Italy & &
\end{tabular}
}
\email{fbracci@mat.uniroma2.it \ herve.gaussier@univ-grenoble-alpes.fr } 

\thanks{$^1$Partially supported by PRIN 2017 Real and Complex Manifolds: Topology, Geometry and holomorphic
dynamics, Ref: 2017JZ2SW5, by GNSAGA of INdAM and by the MIUR Excellence Department Project awarded
to the Department of Mathematics, University of Rome Tor Vergata, CUP E83C18000100006.}
\thanks{$^2$Partially supported by ERC ALKAGE}
\maketitle 

\tableofcontents

\section{Introduction}
Extensions of biholomorphisms between domains is a classical subject which has been attacked and studied by several authors with different techniques. It is virtually impossible to name all authors which contributed to the subject, and we limit ourselves in citing the surveys papers \cite{PSS, Fo, Fo1} and references therein. In some cases, the problem of extension  is dealt with  techniques of  CR geometry and asymptotic expansions of Bergman kernels or other invariant metrics, like the amazing result of Ch. Fefferman \cite{Fe}. In some other cases, extension results are obtained by ``scaling techniques'', essentially introduced in the field of Several Complex Variables by S. Pinchuk \cite{Pi} (see also the survey articles \cite{Be, PSS} and the references therein). In some instances, continuous extension is related to the so-called ``Julia-Wolff-Carath\'eodory type theorems'' (see, {\sl e.g.}, \cite{Aba, BCD, BF}). In most cases, the results are of local nature  and the extension is provided at points which satisfy good conditions (see, {\sl e.g.}, \cite{FR, Le1, Le2, Le3}). 

In this survey, we consider extension property which comes from ``visible geometry'', as introduced in \cite{BB, Ka, Z1, Z2, Z3, BZ, BM, BG, BGZ1, CL, BNT}. 

For the sake of simplicity, we restrict ourselves to biholomorphic maps between bounded domains in $\C^d$, $d\geq 1$.

One of the first results for bounded domains in $\C^d$ ($d \geq 2$), connecting extension of biholomorphic maps to the boundary behaviour of geodesics, is due to Lempert who provided in \cite{Le1} a complete description of complex geodesisc in strongly convex domains, studied their boundary extensions and re-proved as a consequence the Fefferman extension theorem. This extension result was pursued in \cite{Le4}, still using invariant metrics.

The use of invariant metrics is natural when dealing with the problem of biholomorphism extension. Indeed, by definition, biholomorphic maps are isometries for such metrics and the problem reduces to the extension of isometries, or more generally of 1-Lipschitz maps, in complete metric spaces. The crucial point here is to investigate relevant boundaries to which, almost by contruction, 1-Lipschitz maps extend. Such boundaries are abstract boundaries, in the sense that they are boundaries at infinity. From a Complex Analysis point of view, extension of biholomorphisms between bounded domains in $\C^d$ is generally understood as a homeomorphic extension up to the Euclidean boundaries. The use of abstract boundaries, providing nice extension results, is particularly efficient if one can prove that the considered abstract boundary is homeomorphic to the Euclidean boundary, for a given bounded domain in $\C^d$. We present here several abstract boundaries considered in different papers and, as an application, study the extension of biholomorphic maps.

We thank the referee for several useful remarks and comments which improved the original manuscript.

\section{Abstract boundaries and extension}

Let $D \subset\subset \C^d$ be a domain. An {\sl abstract boundary} $\partial_\ast D$ for $D$ is a topological space such that $\overline{D}^\ast:=D\cup \partial_\ast D$, the disjoint union of $D$ and $\partial_\ast D$, has a topology which, restricted to $D$, coincides with the Euclidean topology while, restricted to $\partial_\ast D$ coincides with the topology of $\partial_\ast D$. 

\begin{definition}
Let $\mathcal D$ be a collection of bounded domains in $\C^d$. An abstract boundary is {\sl suitable for $\mathcal D$} if, given any biholomorphism $F:D\to D'$, with $D, D'\in \mathcal D$ then there exists a homeomorphism $F^\ast: \overline{D}^\ast \to \overline{D'}^\ast$ such that $F^\ast|_D=F$. 
\end{definition}

\begin{example}
The Euclidean boundary of a bounded domain in $\C^d$ is an abstract boundary, which is however not suitable  for the collection of all bounded domains in $\C^d$.

The one-point compactification of a bounded domain in $\C^d$ defines an abstract boundary (formed by one point $\{\infty\}$) which is suitable for the collection of all bounded domains in $\C^d$.

\end{example}

Since we want to use abstract boundaries to study extension of biholomorphic maps, the previous example shows that suitable abstract boundaries are too general for this aim. Therefore, among suitable  abstract boundaries we select those that, for some model domains, are equivalent to the Euclidean closure of the domain. 

For a bounded domain $D \subset \C^d$, we denote by $\overline{D}$ the Euclidean closure of $D$. In particular,  $\overline{\B^d}$ denotes the Euclidean closure of the unit ball $\B^d:=\{z\in \C^d: \|z\|<1\}$ in $\C^d$.

\begin{definition}
Let $\mathcal D$ be a collection of bounded domains in $\C^d$ which contains $\B^d$. A {\sl faithful} abstract boundary is an abstract boundary suitable for $\mathcal D$ such that  the identity map ${\sf id}: \B^d \to \B^d$ extends as a homeomorphism ${\sf id}_{\B^d}^\ast: \overline{\B^d}^\ast \to \overline{\B^d}$.
\end{definition}

The choice of considering the ball as ``model domain'' is natural, but in general one has the following:

\begin{definition}
Let $\mathcal D$ be a collection of bounded domains in $\C^d$. A domain $D\in\mathcal D$ is a {\sl model domain} for an abstract boundary suitable for $\mathcal D$ if  the identity map ${\sf id}: D\to D$ extends as a homeomorphism ${\sf id}_{D}^\ast: \overline{D}^\ast \to \overline{D}$.

 A domain $D\in\mathcal D$ is a {\sl quasi-model domain} for an abstract boundary suitable for $\mathcal D$ if  the identity map ${\sf id}: D\to D$ extends as a sequentially continuous surjective map ${\sf id}_{D}^\ast: \overline{D}^\ast \to \overline{D}$.
\end{definition}

Trivially, we have the following extension result:

\begin{theorem}\label{Thm:abstract-ext}
If $D, D'\subset\subset \C^d$ are model domains for an abstract boundary suitable for a collection of domains containing $D$ and $D'$, then every biholomorphism $F:D\to D'$ extends as a homeomorphism from $\overline{D}$ to $\overline{D'}$. The extension of $F$ is given by ${\sf id}_{D'}^\ast\circ F^\ast \circ ({\sf id}_{D}^\ast)^{-1}$.
\end{theorem}

The previous theorem has a weaker version:

\begin{theorem}\label{Thm:abstract-ext2}
Let  $\mathcal D$ be a collection of bounded domains in $\C^d$ and let $D, D'\in \mathcal D$. If $D$ is a model domain and $D'$ is a quasi-model domain for an abstract boundary suitable for  $\mathcal D$,  then every biholomorphism $F:D\to D'$ extends continuously from $\overline{D}$ to $\overline{D'}$. The extension of $F$ is given by ${\sf id}_{D'}^\ast\circ F^\ast \circ ({\sf id}_{D}^\ast)^{-1}$.
\end{theorem}
\begin{proof}
The only subtle point here is that ${\sf id}_{D'}^\ast$ is in principle only sequentially continuous. However, this implies that ${\sf id}_{D'}^\ast\circ F^\ast \circ ({\sf id}_{D}^\ast)^{-1}:{\overline{D} \to \overline{D'}}$ is sequentially continuous. SInce $\overline{D}$ is first countable, then ${\sf id}_{D'}^\ast\circ F^\ast \circ ({\sf id}_{D}^\ast)^{-1}:{\overline{D}\to \overline{D'}}$ is  continuous.
\end{proof}

\begin{remark}
The previous theorems have a converse implication. Let  $\mathcal D$ be a collection of bounded domains in $\C^d$ and let $D, D'\in \mathcal D$. If $D$ is a model domain for some  abstract boundary suitable for $\mathcal D$ and $F:D \to D'$ extends as a homeomorphism ({\sl respectively}, as a continuous surjective map) from $\overline{D}$ to $\overline{D'}$, then $D'$ is a model ({\sl resp.} quasi-model) domain for the same abstract boundary. The extension of ${\sf id}_{D'}$ is given by $ F^\ast\circ  {\sf id}_D^\ast\circ {F^\ast}^{-1}$.
\end{remark}

From the previous considerations is then clear that extension of biholomorphisms is related to the existence of suitable abstract boundaries and model/quasi-model domains. 

In the next section we are going to recall some constructions of abstract boundaries.

\section{The Carath\'eodory boundary} Here we consider $\mathcal D$ to be the collection of all bounded\footnote{actually, one can remove the hypothesis of boundedness by considering hyperbolic simply connected domains in the Riemann sphere, replacing the Euclidean distance with the spherical distance, see, {\sl e.g.}, \cite[Chapter 4]{BCD} for details.} 
 simply connected domain in $\C$. Let $D\in \mathcal D$. A {\sl cross-cut} for $D$ is a continuous injective curve $\gamma:[0,1]\to D$ such that $\gamma((0,1))\subset D$ and $\gamma(0), \gamma(1)\in \partial D$. A cross-cut divides $D$ into two simply connected components. A {\sl null chain} $(C_n)_{n\geq 0}$ is a sequence of cross-cuts such that $\hbox{diam}(C_n)\to 0$ as $n\to \infty$ and $C_{n+1}$ is contained in the connected component of $D\setminus C_n$ which does not contain $C_0\cap D$, for every $n\geq 1$. 
 
 The connected component of $D\setminus C_n$, $n\geq 1$, which does not contain $C_0$ is called the {\sl interior part} of $C_n$. 
 
Let $(C_n)$ and $(G_n)$  be two null chains of $D$ and let $\{V_n\}$ and $\{W_n\}$ denote their interior parts. We say that $(C_n)$ and $(G_n)$ are {\sl equivalent} if for every $n\geq 1$ there exists $m_n$ such that $V_j\subset W_n$ and $W_j\subset V_n$ for all $j\geq m_n$. 

\begin{definition}
The equivalence class $[(C_n)]$ of a null-chain $(C_n)$ in $D$ is called a {\sl prime end}. The {\sl Carath\'eodory boundary} $\partial_C D$ of $D$ is the set of all prime ends of $D$.
\end{definition}

Let $U$ be an open set in $D$.  We define $U^\ast$ to be the union of $U$ and of all the prime ends $[(C_n)]$ such that $\{C_n\cap D\}$ is eventually contained in $U$. We give a topology on $\overline{D}^C:=D\cup \partial_C D$ by considering the topology generated by all open sets $U$ of $D$ and $U^\ast$. 

Note that, if $[(C_n)]$ is a prime end, then $\{V_n\}$ is a countable fundamental system of open neighborhoods of $[(C_n)]$. Carath\'eodory's famous extensions theorems can be translated as follows (see, {\sl e.g.} \cite[Chapter 4]{BCD}):

\begin{theorem}
Let $\mathcal D$  be the collection of all bounded simply connected domains in $\C$. Then the Carath\'eodory boundary is a faithful abstract boundary suitable for $\mathcal D$. Moreover, $D\in \mathcal D$ is a semi-model domain for the Carath\'eodory boundary if and only if $\partial D$ is locally connected, while, $D\in \mathcal D$ is a model domain for the Carath\'eodory boundary if and only if $\partial D$ is a Jordan curve.
\end{theorem}

\section{The Gromov boundary} Let $\mathcal D$  be the collection of all bounded complete (Kobayashi) hyperbolic domains in $\C^d$. Let $D\in \mathcal D$. The completeness of the metric guarantees, via the Hopf-Rinow theorem, that every two points in $D$ can be joined by a geodesic {segment} (namely, a hyperbolic-length minimizing curve).  Let $K_D$ denote the Kobayashi distance of $D$. A {\sl geodesic ray} $\gamma:[0,+\infty)$ with base point $z_0\in D$ is a continuous curve such that for every $t, s\geq 0$,
\[
|t-s|=K_D(\gamma(s), \gamma(t)).
\]
Two geodesic rays $\gamma, \eta$ in $D$ with base point $z_0$ are {\sl asymptotic} if there exists $C>0$ such that $K_D(\gamma(s), \eta(s))\leq C$ for all $s\geq 0$. Being asymptotic is an equivalence relation. Therefore, the following definition makes sense:
\begin{definition}
The {\sl Gromov boundary} $\partial_G D$ of $D$ is the set of all equivalence classes of asymptotic geodesic rays with base point $z_0\in D$.
\end{definition}

We say that a sequence $\{\sigma_n\}\subset \partial_G D$ converges to $\sigma\in \partial_G D$ if, for every $n\in \N$ there exists  a geodesic ray $\gamma_n$  representing $\sigma_n$ such that every subsequence of $\{\gamma_n\}$ contains a subsequence that converges uniformly on compacta of $[0,+\infty)$ to a geodesic ray $\gamma$ such that $\sigma=[\gamma]$. This definition allows to define a topology on $\partial_G D$, which we call the {\sl Gromov topology} (see, {\sl e.g.} \cite[Chapter III.H.3]{BH}). 

The choice of the base point $z_0$ is irrelevant, for, changing the base point naturally defines a homeomorphism on the corresponding Gromov boundaries. 

We define the Gromov closure $\overline{D}^G$ of $D$ by
\[
\overline{D}^G:=D\cup \partial_G D.
\]
We define a topology on $\overline{D}^G$, called the {\sl Gromov topology}, which makes $\partial_G D$ an abstract boundary. To this aim, fix a base point $z_0\in D$. For $z\in D$, let $c(z)$ be a geodesic {segment} joining $z_0$ and $z$---there exists at least one, but it might not be unique. 

Note that, if $\{z_n\}\subset D$ is a sequence, since $(D, K_D)$ is complete, $K_D(z_0,\cdot)$ is proper and, as $c(z_n)$ is an isometry for every $n$, then $\{c(z_n)\}$  is equibounded and equicontinuous on compacta of $[0,+\infty)$, thus, by the Arzel\'a-Ascoli theorem, one can extract a subsequence of $\{c(z_n)\}$ converging uniformly on compacta to a geodesic {segment} or to a geodesic ray (the second possibility occurs if and only if $\{z_n\}$ is compactly divergent in $D$). Moreover, $\{z_n\}\subset D$ converges to $z\in D$ if and only if every converging subsequence of $\{c(z_n)\}$ converges uniformly  to a geodesic {segment} of the type $c(z)$. 

In this way, we can define a topology on $\overline{D}^G$ which coincides with the Euclidean topology on $D$ and with the Gromov topology previously defined on $\partial_G D$. Note that a sequence $\{z_n\}\subset D$ converges in this topology to $\sigma\in \partial_G D$ if for every $n\in \N$ there exists a geodesic $c(z_n)$ such that $\{c(z_n)\}$ converges uniformly on compacta to a geodesic ray $\gamma\in \sigma$. 

By construction, the Gromov boundary is suitable for $\mathcal D$ since biholomorphisms are isometries for the Kobayashi distance. We have

\begin{theorem}
Let $\mathcal D$  be the collection of all bounded complete (Kobayashi) hyperbolic domains in $\C^d$. Then the Gromov boundary is a faithful abstract boundary suitable for $\mathcal D$.
\end{theorem}

There is no general characterization of model and semi-model domains for the Gromov boundary, but they are known in some special cases. 

Let $D$ be a bounded complete hyperbolic domain in $\C^d$. A {\sl geodesic triangle} is the union of three geodesics $\gamma_1, \gamma_2, \gamma_3$ such that the initial point of $\gamma_j$ coincides with the final point of $\gamma_{j+1}$, $j=1,2,3 \mod 3$. The geodesics $\gamma_j$ are called {\sl edges} of the geodesic triangle.

The domain $D$  is called {\sl Gromov hyperbolic} if  there exists $M>0$ such that every edge of every geodesic triangle is contained in the $M$-hyperbolic neighborhood of the union of the other two edges. Note that Gromov hyperbolicity is invariant under biholomorphisms.

The following domains are model domains for the Gromov boundary:

\begin{enumerate}
\item $C^2$-smooth bounded strongly pseudoconvex domains \cite{BB}.
\item Gromov hyperbolic bounded convex domains \cite{BGZ1} (in particular, by \cite{Z1}, smooth bounded convex domains of finite type).
\item Smooth bounded pseudoconvex domains of D'Angelo finite type in $\C^2$ \cite{Fia}.
\end{enumerate}

In particular, since by \cite{BB} $C^2$-smooth bounded strongly pseudoconvex domains  are Gromov hyperbolic ({\sl respectively} by \cite{Fia} smooth bounded pseudoconvex domains of D'Angelo finite type in $\C^2$), (2) implies that every bounded convex domain biholomorphic to a smooth bounded strongly pseudoconvex domain or to a smooth bounded pseudoconvex domain of D'Angelo finite type in $\C^2$ is a model domain for the Gromov boundary. Also, by \cite[Corollary 1.7]{BGNT}, every bounded domain $D\subset \C^d$ such that for every $p\in \partial D$ there exists an open neighborhood $U_p$ of $p$ such that $U_p\cap D$ is a model domain for the Gromov boundary, is a model domain for the Gromov boundary.

To understand semi-model domains for the Gromov boundary, one needs to introduce the notion of {\sl visibility} \cite{BZ, BM, BNT}:

\begin{definition}
A bounded complete hyperbolic domain $D\subset \C^d$ is {\sl visible} if, for every two sequences $\{z_n\}, \{w_n\}\subset D$ such that $\{z_n\}$ converges to a point $p\in \partial D$ and $\{w_n\}$ converges to a point $q\in \partial D$, with $p\neq q$, there exists a compact set $K\subset\subset D$ such that every geodesic in $D$ joining $z_n$ and $w_n$ intersects $K$.
\end{definition}

We have the following result (cf. \cite[Thm. 3.3]{BNT}):

\begin{theorem}
Let $D\subset \C^d$ be a bounded complete hyperbolic domain with no non-trivial analytic discs on $\partial D$. If $D$ is visible then it is a semi-model domain for the Gromov boundary.
\end{theorem}
\begin{proof}
Let $D\subset \C^d$ be a bounded complete hyperbolic domain  with no non-trivial analytic discs on $\partial D$. Assume $D$ is visible.

By \cite[Lemma 3.1]{BNT}, every geodesic ray {\sl lands} ({\sl i.e.}, if $\gamma:[0,+\infty)\to D$ is a geodesic ray, then there exists a point $p\in \partial D$ such that $\lim_{t\to+\infty}\gamma(t)=p$) -- this follows also from D'Addezio's Lemma~\ref{Prop:Dad} taking into account that by hypothesis $D$ does not have analytic discs on the boundary.

Since $\partial D$ does not contain analytic discs, it follows from D'Addezio's Lemma~\ref{Prop:Dad} that if $\gamma$ and $\eta$ are two geodesic rays which are asymptotic, then they land at the same boundary point. Therefore, the map
\[
\Phi_D: \partial_G D \to \partial D
\]
that associates, at every $\sigma\in \partial_G D$, the landing point of any of its representative, is well defined. 

Extend the map $\Phi_D$ in $D$ by declaring $\Phi_D(z)=z$ for all $z\in D$. By \cite[Lemma 3.1]{BNT}, if $\{z_k\}\subset D$ is a sequence converging to a point $p\in \partial D$, and $z_0\in D$, then from every sequence  $\{\gamma_n\}$ of geodesics such that $\gamma_n$ joins $z_0$ to $z_n$, one can extract a subsequence with converges uniformly on compacta to a geodesic ray $\gamma$ starting from $z_0$ and landing at $p$. Therefore, $\Phi_D:\overline{D}^G \to \overline{D}$ is a surjective extension of ${\sf id}_D:D \to D$. It follows again by \cite[Lemma 3.1]{BNT} that $\Phi_D$ is sequentially continuous.
\end{proof}

\begin{remark}
It is an open question if in general the existence of a non-trivial analytic disc on the boundary of a bounded domain prevents visibility. 
\end{remark}

{
We saw in this Section the importance of Gromov hyperbolicity, that can be seen as a condition of negative curvature. There are few characterizations, or even examples, of Gromov hyperbolic domains in $\mathbb C^d$ (see some examples here above) and there are interesting open questions relating different possible notions of negative curvature.

\vspace{1mm}
\noindent{\bf Domains satisfying a uniform squeezing property.}

\vspace{1mm}
We recall that, following \cite{Ye}, a bounded domain $D \subset \C^d$ satisfies a uniform squeezing property if there exists $0 < r < 1$ and, for every $z \in D$, there exists a holomorphic embedding $\varphi_z : D \rightarrow \B^d$ such that $\varphi_z(z) = 0$ and $B(0,r):=\{w \in \C^d /\ \|z\| < r\} \subset \varphi_z(D)$. We define the squeezing function
$$
s_D(z):=\sup \{0 < r < 1/\ \exists \ \varphi_z : D \stackrel{\rm hol. \ emb.}{\longrightarrow} \B^d,\ \varphi_z(z) = 0,\ B(0,r) \subset \varphi_z(D)\}.
$$

Bounded convex domains (without boundary regularity assumption) are examples of domains satisfying a uniform squeezing property. This is the case of the polydisc, which is not Gromov hyperbolic. Strongly pseudoconvex domains, which are Gromov hyperbolic, are examples of domains with squeezing function tending to one at the boundary.

Good estimates of invariant metrics rely only on compactness arguments: the Bergman metric, the K\"ahler-Einstein metric, the Carath\'eodory metric and the Kobayashi metric are consequently all bi-Lipschitz in domains satisfying a uniform squeezing property (see \cite{Ye}). Riemannian manifolds with negative curvature possess very nice metric properties. In the context of K\"ahler manifolds, all the information should be carried over by holomorphic (bi)sectional curvature. To obtain precise estimates, for instance on holomorphic (bi)sectional curvatures of a complete K\"ahler metric on a given domain, one generally needs to osculate the boundary of the domain locally by special domains, obtained essentially under a scaling process, and for which such curvature estimates are known. In the context of squeezing property, this is given by the condition that the squeezing function tends to one at the boundary, such domains being heuristically exhausted by balls; they consequently share many properties of the ball. For instance, they admit a complete K\"ahler-Einstein metric whose holomorphic (bi)sectional curvature converges to the one of the unit ball. The same properties are also satisfied for the Bergman metric, the proof relying for both metrics on a scaling process. In connection with the Gromov hyperbolicity, the following question makes sense, in view of the non Gromov hyperbolocity of the polydisc:

\vspace{1mm}
{\sl Let $D$ be a bounded domain in $\C^d$ for which $s_D(z) \rightarrow 1$ when $z \rightarrow \partial D$. Is $D$ Gromov hyperbolic?}

\vspace{1mm}
Note that we defined Gromov hyperbolicity for a domain in $\C^d$ endowed with its Kobayashi distance. We could have considered any other invariant metric, in the context of complex manifolds. It would be interesting to figure out under which geometric conditions on a given domain, we would obtain the same result by considering any of the classical invariant metrics. For domains satisfying a uniform squeezing property, choosing any of the Bergman, the K\"ahler-Einstein, the Carath\'eodory or the Kobayashi metrics will give the same conclusion, the metrics being all bi-Lipschitz.  

Quite surprisingly the above question is still open. With the previous discussion, we would expect $D$ to be Gromov hyperbolic, since the unit ball is. A strategy to prove it might be as follows. Since the Bergman metric has negatively pinched holomorphic bisectional curvature near the boundary of $D$, there should exist a neighborhood $U$ of $\partial D$ such that for all points $x,y,z \in D \cap U$, sufficiently close to each other, geodesic triangles connecting $x$, $y$ and $z$, for the Bergman metric, should be $\delta$-thin, where $\delta > 0$ depends only on $D$ and on the relative distances between the three points. This should come from the theory of negatively curved Riemann manifolds, applied with the Bergman distance. It follows that $D$ should satisfy a ``local Gromov hyperbolic condition'':

\vspace{2mm}
{\sl There exists $M > 0$ and, for every point $p \in \overline{D}$, there exist two open sets $U$ and $V$ such that $p \in V \subset \subset U$ and for every $x, y, z \in D \cap V$, every geodesic triangle (for the Kobayashi metric) connecting $x$, $y$ and $z$ is contained in $U$ and is $M$-thin.}

\vspace{2mm}
Note that this local property concerns in fact only points $p \in \partial D$. Indeed, since $D$ is uniformly squeezing, $D$ is Kobayashi complete hyperbolic. Fix $p \in D$ and let $V:=\{w \in D/\ K_D(p,w) < 1\}$ and $U:=\{w \in D/\ K_D(z,w) < 2\}$. Then $V \subset \subset U \subset \subset D$ and for all points $w, w' \in V$, every geodesic segment joining $w$ to $w'$ is contained in $U$. Finally, for every triple $(x,y,z)$ of points in $V$, every geodesic triangle connecting these points is $4$-thin.

The matter is finally to go from small triangles to large triangles, meaning from a local property to a global one and to understand the asymptotic behaviour of geodesic segments joining points $x$ and $y$ that are sufficiently close to $\partial D$ but that are far from each other. This is the key argument in \cite{BB} to treat the case of strongly pseudoconvex domains. This is not understood yet for domains with squeezing function converging to one at the boundary. This is partly connected to the visibility condition and one may ask:

\vspace{1mm}
{\sl Question. Does a domain $D$ for which $s_D(z) \rightarrow 1$ when $z \rightarrow \partial D$ satisfy a visibilty property?
}

\vspace{1mm}
Note that there exists an increasing sequence of bounded convex domains $D_{\nu} \subset \mathbb C^2$, converging to the bidisc $\mathbb D^2$, such that for every $\nu \geq 1$, $(D_{\nu},K_{D_{\nu}})$ is Gromov hyperbolic, whereas $(\mathbb D^2,K_{\mathbb D^2})$ is not Gromov hyperbolic. One can take $D_{\nu} : = \{z_1,z_2) \in \mathbb C^2/\ |z_1|^2 +|z_2|^{2\nu} < 1\}$.

There also exists an increasing sequence of bounded convex domains $\Omega_{\nu} \subset \mathbb C^2$, converging to the unit ball, such that for every $\nu \geq 1$, $(\Omega_{\nu},K_{\Omega_{\nu}})$ is not Gromov hyperbolic, whereas $(\mathbb B^2,K_{\mathbb B^2})$ is Gromov hyperbolic. Such domains $\Omega_{\nu}$ can be constructed as polyhedrons and, as in the polydisc case, should not satisfy the visibility property.

These two examples are not surprising due to the characterization of Gromov hyperbolic, bounded convex domains, with smooth boundary, given by A. Zimmer \cite{Z1}: we just considered domains of finite type in the first case and of infinite type in the second one. The precise link between different types of curvature, such as the Gromov hyperbolicity, which refers to an abstract boundary at infinity, and the D'Angelo type, which refers to the Euclidean boundary, is still mysterious as can be seen from the considerations here under.

\vspace{2mm}
\noindent{\bf Equivalence of curvature properties.}

\vspace{2mm}
 We are interested here in understanding the possible links between the Gromov hyperbolicity (metric invariant), the finiteness of the  D'Angelo type (CR invariant) and the negative pinching of the holomorphic bisectional curvature (K\"ahler invariant). The case of pseudoconvex domains is still mysterious, even in complex dimension 2 and the links could be as follows:

\vspace{2mm}
{\sl Conjecture. Let $D$ be a bounded pseudoconvex domain in $\mathbb C^2$, with smooth boundary. Then the following are equivalent:

\begin{enumerate}
\item[(i)] $(D,K_D)$ is Gromov hyperbolic,

\item[(ii)] $D$ is of finite D'Angelo type,

\item[(iii)] $D$ admits a complete K\"ahler metric with holomorphi bisectional curvature negatively pinched near $\partial D$.
\end{enumerate}
}

In $\C^d$, with $d \geq 3$, the situation is more complicated and the above conjecture is probably false for bounded pseudoconvex domains with smooth boundary. Indeed, according to a result of G. Herbort \cite{He} and of J.E. Fornaess - F. Reng \cite{FoRo}, there exists a bounded pseudoconvex domain $D \subset \mathbb C^3$, with smooth boundary, such that the Kobayashi metric is not bi-Lipschitz to any Riemannian metric in $D$. It follows that $D$ does not satisfy a uniform squeezing property according to \cite{Ye} and $D$, given in \cite{FoRo} does not admit any complete K\"ahler metric with negatively pinched holomorphic bisectional curvature, according to a result of D. Wu - S. T. Yau \cite{WY}. Indeed, under such an assumption the Kobayashi metric and the K\"ahler-Einstein metric on $D$ should be bi-Lipschitz. The domain $D$ would be a counterexample to the Conjecture, in higher dimension, if one could prove that the result of Wu - Yau were still valid when replacing ``a complete K\"ahler metric with negatively pinched holomorphic sectional curvature'' by ``a complete K\"ahler metric with negatively pinched holomorphic sectional curvature near the boundary''.

\vspace{1mm}
In the case of convex domains, the following is proved in \cite{BGZ2}:

\begin{theorem}
Let $D \subset \C^d$ be a bounded convex domain, with smooth boundary. We assume that $D$ admits a complete K\"ahler metric with holomorphic bisectional curvature negatively pinched near $\partial D$. Then $D$ is of finite D'Angelo type (and equivalently Gromov hyperbolic).
\end{theorem}
}

\section{The Horosphere boundary}

This construction was introduced in \cite{BG}. Let $\mathcal D$  be the collection of all bounded complete (Kobayashi) hyperbolic domains in $\C^d$. Let $D\in \mathcal D$. 

\begin{definition}
Let $z_0\in D$. A sequence $\{u_n\}\subset D$ is {\sl admissible} if
\begin{enumerate}
\item $\lim_{n\to \infty} K_D(z_0, u_n)=+\infty$,
\item for every $R>0$ the set 
\[
E_{z_0}^D(\{u_n\}, R):=\left\{z\in D: \limsup_{n\to \infty}[K_D(z, u_n)-K_D(z_0, u_n)]<\frac{1}{2}\log R\right\},
\]
is non-empty.
\end{enumerate}
\end{definition}

Two admissible sequences $\{u_n\}, \{v_n\}\subset D$ are {\sl equivalent} provided that, for every $R>0$, there exists $R'>0$ such that
\[
E_{z_0}^D(\{v_n\}, R')\subset E_{z_0}^D(\{u_n\}, R), \quad E_{z_0}^D(\{u_n\}, R')\subset E_{z_0}^D(\{v_n\}, R).
\]
\begin{definition}
The {\sl horosphere boundary} $\partial_HD$ of $D$ is the set of all equivalence classes of admissible sequences. 
\end{definition}

We give on $\overline{D}^H:=D\cup \partial_H D$ a topology which makes $\partial_H D$ an abstract boundary as follows. For every open set $U\subset D$, we define $U^\ast$ to be the union of $U$ and of all $\sigma\in \partial_H D$ such that there exists $\{u_n\}\in \sigma$ so that $E_{z_0}^D(\{u_n\}, R)\subset U$ eventually. The topology of $\overline{D}^H$ is the one generated by open sets $U\subset D$ and $U^\ast$. Changing the base point $z_0$, one obtains an equivalent topology.

In \cite{BG}, it is proved that

\begin{theorem}
Let $\mathcal D$  be the collection of all bounded complete hyperbolic domains in $\C$. Then the horosphere boundary is a faithful abstract boundary suitable for $\mathcal D$.
\end{theorem}

Bounded smooth strongly pseudoconvex domains are model domains for the horosphere boundary. In general, see \cite[Section 7]{BG}, the horosphere boundary is not homemorphic to the Gromov boundary of a domain. 

\section{The Busemann (or horofunctions) boundary} Let $\mathcal D$  be the collection of all bounded complete (Kobayashi) hyperbolic domains in $\C^d$. Let $D\in \mathcal D$. Let $C(D)$ be the space of all continuous functions on $D$ (with real values) endowed with the topology of uniform convergence on compacta. Let $C^\ast(D)$ be the quotient space of $C(D)$ with the subspace of constant functions, endowed with its natural topology. For $f\in C(D)$, let $[f]$ denotes its image in $C^\ast(D)$. 

Let $z_0\in D$. Let $C_{z_0}(D):=\{f\in C(D): f(z_0)=0\}$ endowed with the topology induced by $C(D)$. Then, the map $C^\ast(D)\ni [f]\mapsto f-f(z_0)\in C_{z_0}(D)$ is a homeomorphism. It follows that $\{[f_n]\}\subset C^\ast(D)$ converges to $[f]\in C^\ast(D)$ provided  $\{f_n-f_n(z_0)\}$ converges uniformly on compacta of $D$ to $f-f(z_0)$. 

There is a natural embedding $\iota: D\to C^\ast(D)$ obtained by $\iota(z):=[K_D(z,\cdot)]$. Let $\overline{D}^B:=\overline{\iota(D)}$, where the closure is taken in $C^\ast(D)$, and let $\partial_B D:=\overline{D}^B\setminus\iota(D)$. 

\begin{definition}
The topological space $\partial_B D$ is the {\sl Busemann boundary} of $D$. 
\end{definition}

Note that $\sigma\in \partial_B D$ if there exists a sequence $\{z_n\}\subset D$ such that $\lim_{n\to \infty}[K_D(z_n, \cdot)-K_D(z_0, z_n)]$ converges to $f-f(z_0)$, with $[f]=\sigma$. 

The following result holds:

\begin{theorem}
Let $\mathcal D$  be the collection of all bounded complete hyperbolic domains in ${\C^d}$. Then the Busemann boundary is a faithful abstract boundary suitable for $\mathcal D$.
\end{theorem}

Given a geodesic ray $\gamma:[0,+\infty)\to D$, one can define an element of $\partial_B D$ by considering the so-called {\sl Busemann function} (see, {\sl e.g.}, \cite{BH})
\[
f_\gamma(z):=\lim_{t\to +\infty}(K_D(\gamma(t), z)-t),
\]
which defines a point of $\partial_B D$. Note that if $\gamma, \eta$ are two asymptotic geodesic rays, then $[f_\gamma]=[f_\eta]$. Hence, there is a well defined map, the {\sl Busemann map}, $B:\partial_G D \to \partial_B D$.

A bounded complete hyperbolic domain  has {\sl approaching geodesics} if every two asymptotic geodesic rays $\gamma, \eta$ starting from the same base point has the property that
\[
\lim_{t\to +\infty}K_D(\gamma(t), \eta(t))=0.
\]

In \cite[Thm. 1]{AFGG} it is proved

\begin{theorem}\label{thm:AFGG}
Let $D\subset \C^d$ be a bounded complete hyperbolic domain with approaching geodesics.  If $D$ is Gromov hyperbolic then the Busemann map extends the identity map from $D$ to $D$ to a homeomorphism from $\overline{D}^G$ to $\overline{D}^B$. 
\end{theorem}

{
As mentioned in Section 5, the horosphere closure and the Gromov closure of a bounded complete hyperbolic domain in $\C^d$ may not be homeomorphic. This is for instance the case for the polydisc. Although the Gromov closure can be defined for non Gromov hyperbolic spaces, it seems more accurate, in view of Theorem~\ref{thm:AFGG}, to consider it for Gromov hyperbolic domains in $\C^d$. In \cite{AFGG}, there is an example of a proper geodesic Gromov hyperbolic metric space in $\R^2$ for which the Gromov and the Busemann boundaries are not homeomorphic. It is however not clear that such a phenomenon can occur for complete hyperbolic bounded domains in $\C^d$ endowed with the Kobayashi metric. So the following question is relevant:

{\sl
Let $D$ be a bounded complete (Kobayashi) hyperbolic domain in $\C^d$. Are $\overline{D}^G$, $\overline{D}^H$ and $\overline{D}^B$ all homeomorphic? If not, is the approaching geodesic condition necessary for these abstract boundaries to be homeomorphic ?
}
}
\begin{appendix}
\section{D'Addezio's Lemmas}\label{Daddezio}

The results in this section have been proved by Damiano D'Addezio in his Master Thesis \cite{Dad}.

Let $\Omega\subset \C^d$ be a domain, $k_\Omega$ the infinitesimal Kobayashi (pseudo)metric and, for $z\in \Omega$ denote by $\delta_\Omega(z):=\inf_{\zeta\in\C^d\setminus \Omega}\|z-\zeta\|$. Let

\[
M_\Omega(r):=\sup\left\{\frac{1}{k_\Omega(z;v)}: \delta_\Omega(z)<r, \|v\|=1 \right\}.
\]

\begin{lemma}\label{Lem:Dad}
Let $\Omega\subset \C^d$ be a complete hyperbolic bounded domain. If $\partial \Omega$ does not contain non-trivial analytic discs then
\[
\lim_{r\to 0^+}M_\Omega(r)=0.
\]
\end{lemma}
\begin{proof}
First note that $r\mapsto M_\Omega(r)$ is monotone decreasing. Suppose by contradiction that the statement is false. Then there exists $C>0$ such that $\lim_{r\to 0^+}M_\Omega(r)=C$. Therefore, we can find a sequence of positive real numbers $\{r_n\}$ converging to $0$, points $z_n\in \Omega$ and vectors $v_n\in \C^d$, $\|v_n\|=1$ such that $\delta_\Omega(z_n)\leq r_n$ and 
\[
\frac{1}{k_\Omega(z_n, v_n)}>C-\frac{1}{n}.
\]
By the very definition of $k_\Omega$, there exists a holomorphic function $\varphi_n:\D \to \Omega$ such that $\varphi_n(0)=z_n$ and
\[
\|\varphi_n'(0)\|\geq C-\frac{1}{n}.
\]
Since $\overline{\Omega}$ is compact, up to subsequences, we might assume that there exist $p\in \partial \Omega$ and $\varphi:\D \to \overline{\Omega}$ such that $\lim_{n\to \infty}z_n=p$ and $\{\varphi_n\}$ converges uniformly on compact to $\varphi$. Since $\Omega$ is taut, it follows that $\varphi(\D)\subset \partial \Omega$ and $\|\varphi'(0)\|\geq C$. Therefore, $\varphi(\D)$ is a non-trivial analytic disc in $\partial \Omega$, contradiction.
\end{proof}

\begin{lemma}[D'Addezio]\label{Prop:Dad}
Let $\Omega\subset \C^d$ be a complete hyperbolic bounded domain. Assume that $\partial \Omega$ does not contain non-trivial analytic discs. Then, if $\{z_n\}, \{w_n\}\subset \Omega$ are two sequences such that $\lim_{n\to \infty}z_n=p$, $\lim_{n\to \infty}w_n=q$ and 
\[
\sup_n K_\Omega(z_n, w_n)<+\infty
\]
it follows that $p=q$.
\end{lemma}
\begin{proof}
We argue by contradiction. Assume $p\neq q$. By Lemma~\ref{Lem:Dad}, $\lim_{r\to 0^+}M_\Omega(r)=0$. Let $\gamma_n:[a_n, b_n]\to \Omega$ be a geodesic for $K_\Omega$ such that $\gamma_n(a_n)=z_n$ and $\gamma_n(b_n)=w_n$. 

We claim that there exists a compact set $K\subset \Omega$ such that $\gamma_n([a_n, b_n])\cap K\neq \emptyset$ for all $n$.

Assume the claim is false. Then, for every $r>0$ there exists $n_r$ such that $\delta_\Omega(\gamma_n(t))<r$ for all $t\in [a_n, b_n]$ and $n\geq n_r$. Moreover, since $p\neq q$, we can assume there exists $c>0$ such that $\|z_n-w_n\|>c$ for all $n$. Hence,
\begin{equation*}
\begin{split}
K_\Omega(z_n, w_n)&=\int_{a_n}^{b_n}k_\Omega(\gamma_n(t); \gamma_n'(t))dt=
\int_{a_n}^{b_n}\|\gamma_n'(t)\| k_\Omega(\gamma_n(t); \frac{ \gamma_n'(t)}{\| \gamma_n'(t)\|})dt\\&>\frac{c}{M_\Omega(r)}.
\end{split}
\end{equation*}
Since $\lim_{r \to 0^+}\frac{c}{M_\Omega(r)}=\infty$, it follows that $K_\Omega(z_n, w_n) \to \infty$  as $n\to \infty$, contradiction.

Therefore, there is a compact set $K\subset \Omega$ such that for every $n$ there exists $t_n\in [a_n, b_n]$ so that $\gamma_n(t_n)\in K$. Let $\zeta_0\in \Omega$. We have,
\begin{equation*}
\begin{split}
K_\Omega(z_n, w_n)&=K_\Omega(z_n, \gamma_n(t_n))+K_\Omega(w_n, \gamma_n(t_n))\\&\geq K_\Omega(z_n, \zeta_0)+K_\Omega(w_n, \zeta_0)-2\max\{K_\Omega(\xi, \zeta_0): \xi\in K\}.
\end{split}
\end{equation*}
Since $K_\Omega$ is complete, it follows that $\lim_{n\to \infty}K_\Omega(z_n, w_n)=\infty$,  a contradiction.
\end{proof}

Given a complete hyperbolic domain $D\subset \C^d$, we say that a geodesic ray $\gamma$ of $D$ {\sl lands} if there exists a point $p\in \partial D$ such that $\lim_{t\to +\infty}\gamma(t)=p$. 

As a corollary of the previous lemma, we have the following extension result (cfr. with Abate's Lindel\"of Theorem \cite{Aba}):

\begin{proposition}
Let $\Omega\subset \C^d$ be a complete hyperbolic bounded domain without analytic discs on the boundary. Assume that all geodesic rays in $\Omega$ starting from a point $w_0\in \Omega$ are landing. Let $D\subset \C^d$ be a bounded strongly pseudoconvex domain with $C^2$ boundary. Let $F:D\to \Omega$ be a biholomorphism. Then $F$ has non-tangential limit at every $p\in \partial \B^d$. 
\end{proposition}

\begin{proof}
Let $\gamma$ be a geodesic ray in $D$. Then (see, {\sl e.g.} \cite{BG}), $\gamma$ lands non-tangentially to a point $p\in \partial D$. Moreover, every sequence in $D$ which converges to $p$ stays at finite Kobayashi distance from $\gamma$. 

Take $p\in \partial D$ and let $\gamma$ be a geodesic ray in $D$ converging to $p$ (this exists by \cite{BB}). 

Since $F$ is a biholomorphism, then $F\circ \gamma$ is a geodesic ray in $\Omega$ and hence, by hypothesis, it lands to some point $q\in \partial D$. 

Now, if $\{z_n\}\subset D$ is a sequence converging non-tangentially to $p$, then $\{z_n\}$ stays at finite Kobayashi distance from $\gamma$ and hence, since $F$ is an isometry for the Kobayashi distance, $\{F(z_n)\}$ stays at finite distance from $F\circ \gamma$. By Lemma~\ref{Prop:Dad}, $\{F(z_n)\}$ converges to $q$.
\end{proof}

Note that, if the domain $\Omega$ in the previous proposition were visible, then $\Omega$ has no analytic discs on the boundary and all geodesic rays in $\Omega$ are landing and thus the previous result applies. However, in such a situation, by \cite[Prop. 3.6]{BNT}, $F$ not only has non-tangential limit, but also extends continuously to the boundary of $D$. An example of an application the previous result without having continuous extension is given by a Riemann map from the unit disc whose image is starlike with respect to $0$ but the boundary is not locally connected (see, {\sl e.g.}, \cite{BCD}). 

In higher dimension, L. Lempert \cite{Le5} proved that every biholomorphism from two bounded starlike domains with real analytic boundaries extends as a homeomorphism on the closure of the domains.
\end{appendix}

{}

\end{document}